\documentclass[12pt]{article}

\usepackage[margin=2.5cm]{geometry}

\usepackage{amsmath, amssymb, latexsym, amsthm}
\usepackage{tikz}
\usetikzlibrary{arrows}
\usetikzlibrary{patterns}

\newtheorem{theorem}{Theorem}[section]
\newtheorem{proposition}[theorem]{Proposition}

\newtheorem{corollary}[theorem]{Corollary}
\newtheorem{lemma}[theorem]{Lemma}

\newtheorem*{theorem*}{Theorem}
\newtheorem*{conjecture*}{Conjecture}
\newtheorem*{corollary*}{Corollary}
\newtheorem*{proposition*}{Proposition}
\newtheorem{claim}{Claim}

\newtheorem{problem}[theorem]{Problem}

\theoremstyle{remark}

\begin{document}
\date{}
\title{Brooks' theorem with forbidden colors}

\author{
Carl Johan Casselgren\footnote{Department of Mathematics, 
Link\"oping University, 
SE-581 83 Link\"oping, Sweden.
{\it E-mail address:} carl.johan.casselgren@liu.se.
Research supported by a grant from the Swedish Research council VR
(2017-05077).
}

}

\date{\today}
\maketitle

\bigskip
\noindent
\begin{abstract}
We consider extensions of Brooks' classic theorem on vertex coloring
where some colors cannot be used on certain vertices.
In particular we prove that if $G$ is a connected
graph with maximum degree $\Delta(G) \geq 4$ that is not a complete graph
and $P \subseteq V(G)$ is a set of vertices
where either

(i) at most $\Delta(G)-2$ colors are forbidden for every vertex in $P$,
	and any two vertices of $P$
	are at distance at least $4$, or
	
	(ii) at most $\Delta(G)-3$ colors are forbidden for every vertex in $P$, 
	and any two vertices of $P$
	are at distance at least $3$,

\noindent
then there is a proper $\Delta(G)$-coloring of $G$ respecting these constraints.
In fact, we shall prove that these results hold in the more general setting
of list colorings. These results are sharp.

\end{abstract}


\section{Introduction}

Brooks' classic theorem on graph coloring states that if $G$ 
is a connected graph, not isomorphic to
a complete graph or an odd cycle, then $G$ is properly $\Delta(G)$-colorable,
where $\Delta(G)$ as usual denotes the maximum degree of $G$.
This theorem has been strengthened in many different ways, see e.g.~the recent survey
\cite{CranstonRabern}. In particular, a list coloring version of Brooks' theorem
was obtained already in the seminal paper on list coloring
by Erd\H os et al \cite{ERT}, and also
independently by Vizing \cite{Vizing}.

A {\em precoloring} (or {\em partial coloring}) of a graph $G$ is a proper
coloring of some subset $V' \subseteq V(G)$. Given a precoloring $\varphi$ of $G$, 
we are usually interested in finding an {\em extension} of $\varphi$, that is, a proper coloring
of $G$ that agrees with the partial coloring $\varphi$. If there is such a coloring
(using the same number of colors as $\varphi$), then $\varphi$ is {\em extendable}.

A precoloring extension version of Brooks' theorem was obtained
by Albertson et al \cite{AlbertsonKostochkaWest}
and, independently, by Axenovich \cite{Axenovich}: if $P$ is an independent set, 
$\Delta(G) \geq 3$
and the minimum distance between any two vertices in $P$ is $8$, then any
proper coloring of $P$ can be extended to a proper $\Delta(G)$-coloring of $G$,
unless $G$ contains a copy of $K_{\Delta(G)+1}$, where $K_n$ as usual
denotes a complete graph on $n$ vertices.
In general, the condition on the distance is tight, but Voigt 
obtained improvements for the case when $G$ is $2$-connected \cite{Voigt1, Voigt2}.

In this short note we consider the similar problem of constructing proper $\Delta(G)$-colorings
{\em avoiding} certain colors; that is, given a subset $P \subseteq V(G)$ of the vertices
of a graph $G$ where every vertex of $P$ is assigned a set of 
forbidden colors from $\{1,\dots, \Delta(G)\}$, 
we are interested in finding a proper $\Delta(G)$-coloring
of $G$ respecting these constraints.
These type of questions go back to a paper by H\"aggkvist \cite{Haggkvist}
and although they arise naturally in problems where
a coloring is constructed sequentially,
it seems that they
so far primarily have been studied in the setting of edge colorings, 
see e.g.~\cite{CasselgrenPham, Casselgren, CasselgrenJohanssonMarkstrom,EGHKPS} 
and references therein.
In particular, a variant of Vizing's edge coloring theorem with forbidden colors was obtained in
\cite{EGHKPS}.
Note that it is not possible to prove a variant of K\"onig's edge coloring theorem 
with forbidden colors (even if only one forbidden color is assigned
to edges in a matching),
as the coloring of the graph in Figure \ref{fig:edge}
shows; here $K^{(i)}_{n,n}-e$ denotes a copy of the complete bipartite graph
$K_{n,n}$ where an arbitrary edge
has been removed (which is dashed in the figure). 
Since color $1$ cannot be used on any of the edges at the top
of the figure, it cannot be used on any edge incident with the bottom vertex $u$
either. Hence, there is no proper $\Delta(G)$-edge coloring respecting the forbidden
colors.

Note that here we can make the distance between the colored
edges arbitrarily large by adding more copies of $K_{n,n}-e$ along with
connecting edges between different copies of $K_{n,n}-e$.
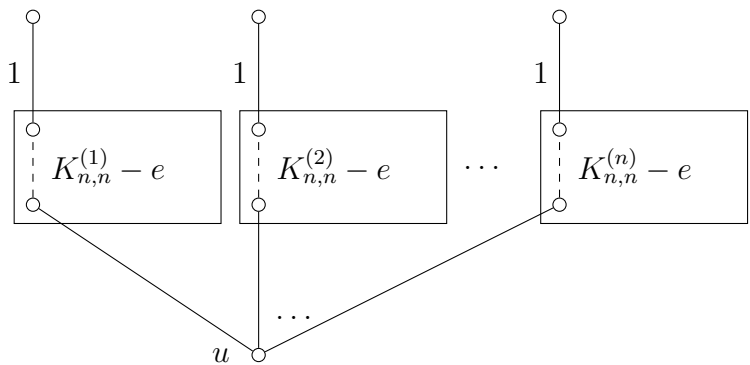
\begin{figure} [h]
	\begin{center}

	\begin{tikzpicture}[scale=0.5]
	\tikzset{vertex/.style = {shape=circle,draw,
												inner sep=0pt, minimum width=5pt}}
	\tikzset{edge/.style = {-,> = latex'}}
	
	\node[vertex] (u) at  (0, -2) {};
	\node at (-1,-2) {$u$};
	\node[vertex] (x_1) at  (-6, 2) {};
	\node[vertex] (x_2) at  (-6, 4) {};
	\node[vertex] (x_3) at  (-6, 7) {};
	\node[vertex] (y_1) at  (0, 2) {};
	\node[vertex] (y_2) at  (0, 4) {};
	\node[vertex] (y_3) at  (0, 7) {};
	\node[vertex] (z_1) at  (8, 2) {};
	\node[vertex] (z_2) at  (8, 4) {};
		\node[vertex] (z_3) at  (8, 7) {};

	\draw[dashed] (x_1) to node [] {}  (x_2);
	\draw[edge] (x_2) to node [midway,left] {$1$} (x_3);
	\draw[edge] (y_2) to node [midway,left] {$1$} (y_3);
	\draw[dashed] (y_1) to node [] {}  (y_2);
	\draw[edge] (z_2) to node [midway,left] {$1$} (z_3);
	\draw[dashed] (z_1) to node [] {}  (z_2);
	\draw[edge] (u) to node [] {}  (x_1);
		\draw[edge] (u) to node [] {}  (y_1);
	\draw[edge] (u) to node [] {}  (z_1);
	
\draw [draw=black] (-6.5,4.5) rectangle (-1,1.5);
\node at (-4,3) {$K^{(1)}_{n,n}-e$};
\draw [draw=black] (-0.5,4.5) rectangle (5,1.5);
\node at (2,3) {$K^{(2)}_{n,n}-e$};
\draw [draw=black] (7.5,4.5) rectangle (13,1.5);
\node at (10,3) {$K^{(n)}_{n,n}-e$};

\node at (6,3) {$\dots$};
\node at (1,-1) {$\dots$};

\end{tikzpicture}
\end{center}
\caption{A bipartite graph $G$ with no proper $\Delta(G)$-edge coloring
respecting the forbidden colors.}
	\label{fig:edge}
\end{figure}

Let us note that the general problem of coloring 
a given graph properly subject to the condition that some
colors cannot be used on certain vertices is certainly NP-complete
since the general vertex (list) coloring problem is.

Here we prove that if $G$ is a graph with $\Delta(G) \geq 4$, 
not containing a copy of $K_{\Delta(G)+1}$,
and $P \subseteq V(G)$ an independent set
where
\begin{itemize}
	
	\item[(i)] at most $\Delta(G)-2$ colors are forbidden for every vertex in $P$,
	and any two vertices of $P$
	are at distance at least $4$, or
	
	\item[(ii)] at most $\Delta(G)-3$ colors are forbidden
	for every vertex in $P$,
	and any two vertices of $P$
	are at distance at least $3$,
	\end{itemize}
then there is a proper $\Delta(G)$-coloring of $G$ respecting these constraints.
It remains an open problem whether (i) holds in 
the case when $\Delta(G) = 3$.

Comparing this to the aforementioned
results on precoloring extension \cite{Albertson,AlbertsonKostochkaWest},
we see that allowing
just the slight extra flexibility of having one or two extra colors on the vertices
in $P$ implies that the distance condition
can be relaxed significantly.

Moreover, as  for the results on precoloring extension in
\cite{Albertson,AlbertsonKostochkaWest}, we shall
prove that these results hold in the more general setting of list coloring of graphs.
This is proved in the next section, where we also give examples
showing that they are sharp.
In Section 3 we briefly consider the related question of {\em avoiding} a given partial 
coloring $\varphi$, i.e.~the problem of finding a coloring $f$ of a graph which differs 
from $\varphi$ on every vertex that is colored under $\varphi$, and
give some observations and remarks on this problem.

\section{Main Result}

	Before proving our main result, let us introduce some terminology.
		The {\em block-cutpoint graph} of a graph $G$ has a vertex for each
	block of $G$ and a vertex for each cut-vertex of $G$, where a cut-vertex
	$v$ is adjacent to a block $B$ if $v \in V(B)$.
	A {\em leaf block} in a graph $G$ is a block of $G$ that 
	contains at most one cut-vertex. 
	We shall use standard graph theory notation; in particular,
	$d_{G}(v)$ denotes the degree of a vertex $v$ in $G$ and an induced subgraph of $G$
	is denoted by $G[S]$, where $S$ is a set of vertices or a set of edges.
	
	Given a graph $G$, assign to each vertex $v$ of $G$ a set
	$L(v)$ of colors (positive integers).
	Such an assignment $L$ is called a \emph{list assignment} for $G$ and
	the sets $L(v)$ are referred
	to as \emph{lists} or \emph{color lists}.
	If all lists have equal size $k$, then $L$
	is called a \emph{$k$-list assignment}.
	Usually, one is interested in finding a proper
	vertex coloring $\varphi$ of $G$,
	such that $\varphi(v) \in L(v)$ for all
	$v \in V(G)$. If such a coloring $\varphi$ exists then
	$G$ is \emph{$L$-colorable} and $\varphi$
	is called an \emph{$L$-coloring}; if $L$ is clear from the context, 
	then we
	just call $\varphi$ a {\em list coloring} and say that $G$ is 
	{\em list colorable} when the coloring exists.
	Furthermore, $G$ is
	called \emph{$k$-choosable} if it is $L$-colorable
	for every $k$-list assignment $L$.
	The least number $k$ such that $G$ is $k$-choosable is called
	the \emph{list-chromatic number} (or {\em choice number}) of $G$ and is denoted by
	$\chi_l(G)$.

As mentioned above, Brooks' theorem holds in the setting of
list coloring; that is, if $G$ is a connected graph which is not isomorphic
to an odd cycle or a complete graph, then $\chi_l(G) \leq \Delta(G)$.
As in the aforementioned papers on precoloring extension,
we shall use the following stronger theorem 
proved in \cite{Borodin, ERT}. A {\em Gallai tree} is a connected
graph in which every block is a complete graph or an odd cycle.
A graph $G$ is {\em degree-choosable} if it has an $L$-coloring
whenever $L$ is a list assignment such that $|L(v)| \geq d_G(v)$
for all $v \in V(G)$; such a list assignment of $G$
is called {\em supervalent}.

\begin{theorem}
\label{th:charac}
	If $L$ is a supervalent list assignment for a connected graph
	$G$ and there is no $L$-coloring of $G$, then
	
	(a) $|L(v)| = d_G(v)$ for every $v \in V(G)$;
	
	(b) $G$ is a Gallai tree;
	
	(c) $L(v) = \cup_{B\in \mathcal{B}(v)} L_B$ for all
	$v \in V(G)$, where $\mathcal{B}(v)$ is the set of blocks containing
	$v$, and for each block $B$, $L_B$ is a set of $\chi(B) -1$ colors.
\end{theorem}

Note that this implies (as remarked in \cite{AlbertsonKostochkaWest})
that each block $B$ is an $|L_B|$-regular graph, and
that all vertices of a single block that are not cut-vertices of $G$
have the same list. Furthermore, we shall need the following.

\begin{corollary}
\label{cor:leaf}
	Suppose that $T$ is a Gallai tree, $L$ a supervalent list assignment
	satisfying conditions (a), (b), and (c) of Theorem \ref{th:charac},
	and that there is a leaf block $B_L$ of $T$ with $\Delta(T) \geq 3$
	vertices $u_1,\dots, u_{\Delta(T)-1}, v$, where $v$ is a cut-vertex
	of $T$ of degree $\Delta(T)$.
	Assume further that $u_0$ is a vertex of degree $\Delta(T)-1$,
	not contained in $B_L$,
	and that there are colors $c_1, c_2,c_3$ such that
	\begin{itemize}
		\item $c_1 \notin L(u_i)$, for $i=0,\dots, \Delta(T)-1$, and
		
		\item $c_2,c_3 \in L(u_i)$, for $i=0,\dots, \Delta(T)-1$.
	\end{itemize}
	Set $U =\{u_0, u_1, \dots, u_{\Delta(T)-1}\}$ and
	define the list assignments $L'$ and $L''$ for $T$ by setting
	\begin{itemize}
		
		\item $L''(v) = L'(v) = L(v)$ if $v \in V(T) \setminus U$, and
	
		\item $L'(v) = L(v)  \cup \{c_1\} \setminus \{c_2\}$
		and $L''(v) = L(v) \cup \{c_1\} \setminus \{c_3\} $ if $v \in U$.
		
	\end{itemize}
	If $T$ is not $L'$-colorable, then it is $L''$-colorable.
\end{corollary}
The proof of this corollary is a rather straightforward application of
Theorem \ref{th:charac}; we briefly sketch an argument.

The conditions imply that there must be
a block $B_2 \neq B_L$ containing $v$, and since $B_L$ is a leaf block
with $\Delta(T)$ vertices, $B_2 \cong K_2$, that is, $B_2=vx$, say. Moreover, since
$T$ is neither $L$-colorable, nor $L'$- or $L''$-colorable, 
$L_{B_2}=\{c_1\}$, $L'_{B_2}=\{c_2\}$, and $L''_{B_2}=\{c_3\}$,
where we use the notation from Theorem \ref{th:charac}.
Thus there must be a block $B_3$ in $T$ containing $x$, such that 
$c_2 \in L_{B_3}$ and $c_1 \in L'_{B_3}$.

Now, since $L$ and $L'$ only differs on the lists for the vertices
in $U$, $B_3 \cong K_2$, and, more generally, there is path
from $v$ to $u_0$ in $T$, all edges of which lie in blocks
that are isomorphic to $K_2$. Furthermore, since $T$ is not $L$-
or $L'$-colorable, every block $B_i$ of this path satisfies
that $L_{B_i} = \{c_1\}$ or $L_{B_i} = \{c_2\}$.
Now, since $T$ is a Gallai tree, this is the only path from
$v$ to $u_0$ in $T$. Using this property,
it is now easy to check that this implies that
$T$ must be $L''$-colorable.

The following is the main result of this paper.

\begin{theorem}
\label{thm:avoidprecol}
	Let $G$ be a connected graph with maximum degree $\Delta(G) \geq 4$
	which is not complete and $P \subseteq V(G)$ a subset of vertices.
	If $L$ is a list assignment for $G$ such that
	$|L(v)| \geq \Delta(G)$ if $v \in V(G)\setminus P$,
	then $G$ is $L$-colorable if either
	
	\begin{itemize}
	
	\item[(i)] $|L(v)| \geq 2$ and any two vertices of $P$
	are at distance at least $4$, or
	
	\item[(ii)] $|L(v)| \geq 3$ and any two vertices of $P$
	are at distance at least $3$.
	\end{itemize}
\end{theorem}

Before proving the theorem, let us demonstrate that it is sharp both with respect to
list sizes and the distance conditions. For the case
$\Delta(G) = 3$, we do not know whether a result as in (i) holds, but the 
example below shows that distance $4$ would be best possible.

That the list size cannot be improved in part (i) follows by
taking two copies $K^{(1)}_{\Delta-1}$  and $K^{(2)}_{\Delta-1}$ of 
$K_{\Delta-1}$, where all vertices have the list $\{1,\dots, \Delta\}$ 
and making exactly one vertex of $K^{(1)}_{\Delta-1}$
adjacent to exactly one vertex of $K^{(2)}_{\Delta-1}$. All other vertices
of $K^{(1)}_{\Delta-1}$ are adjacent to a vertex $u$ with the list $\{1\}$
and all vertices of $K^{(2)}_{\Delta-1}$ are adjacent to a vertex $v \neq u$ with
the list $\{1\}$. 
Clearly, this graph is not list colorable.

\begin{figure} [h]
	\begin{center}

	\begin{tikzpicture}[scale=0.5]
	\tikzset{vertex/.style = {shape=circle,draw,
												inner sep=0pt, minimum width=5pt}}
	\tikzset{edge/.style = {-,> = latex'}}

	\node[vertex] (x_1) at  (-6, 4) {};
	\node[vertex] (x_2) at  (-2, 4) {};
	\node at (-4,4) {$\dots$};
	\node[vertex] (x_3) at  (-4, -1) {};
	\draw (-4,2) circle (3.5);
	\node at (-4,2) {\large $\{1,\dots,\Delta\}$};
	\node at (-9,2) {\Large $K_{\Delta}$};

	\node[vertex] (y_1) at  (4, 6) {};
	\node[vertex] (y_2) at  (8, 6) {};
	\node at (6,6) {$\dots$};
	\node[vertex] (y_3) at  (4, 10) {};
	\draw (6,8) circle (3.5);
	\node at (6,8) {\large $\{1,\dots,\Delta\}$};
	\node at (11,8) {\Large $K_{\Delta}$};

	\node[vertex] (z_1) at  (4, -3) {};
	\node[vertex] (z_2) at  (4, -7) {};
	\node at (4,-5) {$\vdots$};
	\node[vertex] (z_3) at  (8, -3) {};
	\draw (6,-5) circle (3.5);
	\node at (6.7,-5) {\large $\{1,\dots,\Delta\}$};
	\node at (11,-5) {\Large $K_{\Delta}$};

	\node[vertex] (u) at  (-1.5, 8) {};
	\node at (-3,8) {$\{1,2\}$};
	\node[vertex] (v) at  (8, 2) {};
	\node at (9.5,2) {$\{1,2\}$};
	\node[vertex] (w) at  (-1, -3) {};
	\node at (-2.5,-3) {$\{1,2\}$};

	\draw[edge] (u) to node [] {} (x_1);
	\draw[edge] (u) to node [] {} (x_2);
	\draw[edge] (u) to node [] {} (y_3);
	
	\draw[edge] (v) to node [] {} (y_1);	
	\draw[edge] (v) to node [] {} (y_2);
	\draw[edge] (v) to node [] {} (z_3);

	\draw[edge] (w) to node [] {} (z_1);	
	\draw[edge] (w) to node [] {} (z_2);
	\draw[edge] (w) to node [] {} (x_3);


%
\end{tikzpicture}
\end{center}
\caption{A graph with a list assignment and no list coloring.}
	\label{fig:vertex}
\end{figure}
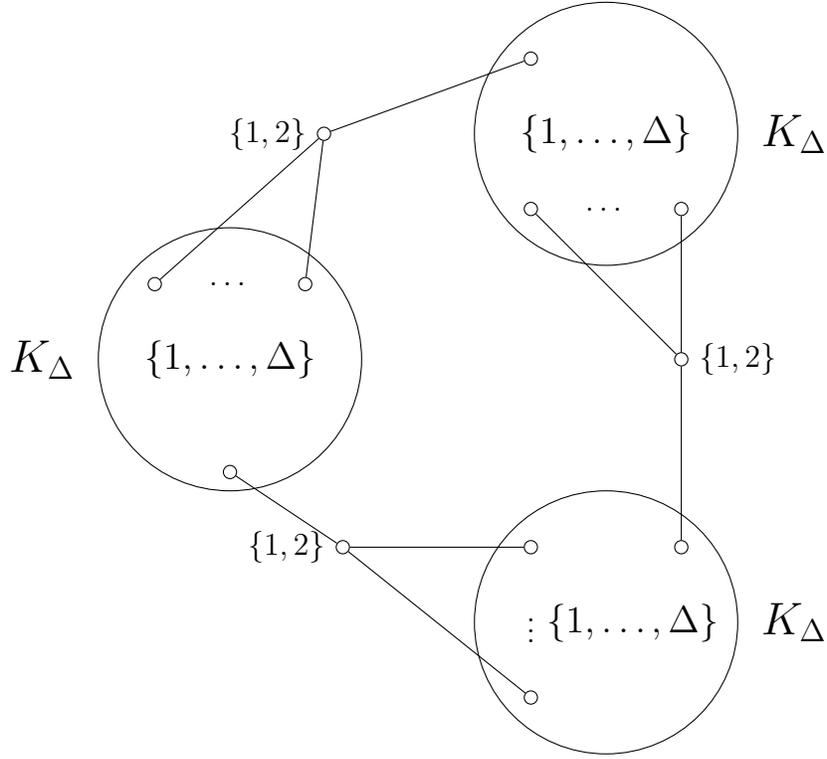

That the distance condition in part (i) is best possible  follows from
the fact that the $\Delta$-regular graph in Figure \ref{fig:vertex} is not
list colorable.
Here all vertices that are contained in copies of $K_{\Delta}$ are assigned the list
$\{1,\dots,\Delta\}$, while the vertices not contained in the copies of $K_{\Delta}$
are assigned the list $\{1,2\}$. Since at least two of the vertices
not contained in the copies of $K_{\Delta}$ must be colored by the same color,
this graph is not list colorable.
Hence, the distance condition in part (i) cannot be improved.
Moreover, this example also shows that the list size in part (ii) is sharp.

That the distance condition in part (ii) is best possible can be seen
by taking a copy of $K_{\Delta-1}$ along with two vertices $u$ and $v$
that are joined to every vertex of $K_{\Delta-1}$ by an edge.
Assign the list $\{1,2,3\}$ to $u$ and the list $\{4,5,6\}$ to $v$ and
the list $\{1,2,\dots, \Delta\}$ to all vertices of $K_{\Delta-1}$.
Then this graph is not list colorable.

Let us now turn to the proof of Theorem \ref{thm:avoidprecol}.
We shall need the following lemma. Recall that an {\em acyclic} subgraph
of a digraph is a subgraph with no directed cycle.

\begin{lemma}
\label{lem:digraph}
	If $D$ is a digraph and $S \subseteq V(D)$ a set of vertices $v$ satisfying
	that the outdegree is strictly larger than the indegree of $v$, then there
	is an ayclic subgraph $D'$ of $D$ such that every vertex of $S$ has outdegree
	$1$ in $D'$.
\end{lemma}
Since every vertex of $S$ has strictly
larger outdegree than indegree, this is straightforward; we leave the details
to the reader.

\begin{proof}[Proof of Theorem \ref{thm:avoidprecol}]
	Since (in both cases) $P$ is independent and $|L(v)| \geq 2$ if $v \in P$, 
	there is an $L$-coloring $\varphi'$ of the vertices in $P$. 
	Define a new list assignment $L'$ for $H=G-P$ by, for $v \in V(H)$, setting 
	$$L'(v) = L(v) \setminus \{\varphi'(u) : \text{$u$ is adjacent to $v$}\}.$$
	Since $|L(v)| \geq \Delta(G)$ if $v\in V(G) \setminus P$,
	it follows from Theorem \ref{th:charac}
	that any component of $H$ that is not $L'$-colorable is a Gallai tree.
	Moreover, since any vertex in $G$ is adjacent to at most one vertex of $P$,
	every component of $H$ that is not $L$-colorable
	has minimum degree $\Delta(G)-1$, and all vertices of such a component
	has degree $\Delta(G)$ in $G$.

	Suppose that 
	some component $T$ of $H$ together with the restriction of $L'$ to $T$ 
	is not list colorable, and thus 
	%
	satisfies the conditions in Theorem \ref{th:charac}. We call such a component
	{\em $\varphi'$-bad} (or just {\em bad}).
	Suppose further that some vertex $p \in P$ is adjacent to exactly $\Delta(G) -1$ vertices
	of a leaf block $B_1 \cong K_{\Delta(G)}$ in $T$ and has no other neighbor in $T$.
	By recoloring $p$ with another color of its list we get a new coloring
	$\varphi''$ of $P$ so that
	$T$ is not $\varphi''$-bad.
	Similarly, if $p \in P$ has only one neighbor in $T$, then by recoloring $p$
	we obtain a coloring $\varphi''$ such that $T$ is not $\varphi''$-bad.
	Our proofs of both part (i) and (ii) relies heavily on this simple observation.
	
	Without loss of generality, we shall assume that every component of $H$ satisfies
	the conditions (a) and (b) of Theorem \ref{th:charac} (with $L'$ in place of $L$); 
	otherwise, we just consider a suitable subgraph of $H$, 
	since components not satisfying these conditions cannot be bad.
	
	\bigskip
	
	After these initial observations, let us now
	prove part (i) of the theorem. 
	Since $G$ is not complete, and any two vertices of $P$ are at 
	distance at least $4$,
	no component of $H$ is isomorphic to a complete graph, and 
	no two vertices of $P$ are adjacent to 
	vertices of the same block in $H$,
	and
	thus no vertex of $P$ is adjacent to vertices of two different leaf blocks.
	This implies that every component of $H$
	contains at least two leaf blocks. Moreover,
	if $T$ is such a component,
	then every leaf block $B_L$ of $T$ is isomorphic to
	$K_{\Delta(G)}$, $\Delta(G)-1$ vertices of $B_L$ are adjacent to the same
	vertex of $P$, and the cut-vertex of $B_L$ is not adjacent to any vertex of $P$.
	
	Now, suppose $\varphi'$ is a coloring of $P$ so that
	some component $T$ of $H$ is $\varphi'$-bad.
	We shall prove that by recoloring some vertices 
	$P' \subseteq P$ satisfying that if $p \in P'$ then $p$ has $\Delta(G)-1$ neighbors
	in a leaf block,
	we can obtain an $L$-coloring $\psi$ of $P$ so that no component of
	$H$ is $\psi$-bad. Indeed, in the following we shall describe a recoloring procedure.
	
	To this end, we shall prove the following claim.
	Let $P_1$ be the set of all vertices of $P$ that have at least one neighbor in
	a leaf block of $H$, and let
	$V_1$ be the set of all vertices in $H$ that are not contained in a leaf-block
	and has a neighbor in $P_1$.
	
	\begin{claim}
	\label{cl:bad}
			If $T$ is a component in $H$, then the number of leaf-blocks in $T$
			is greater than the number of vertices in $T$ that are in $V_1$.
	\end{claim}
\begin{proof}
	Let us first note that no block of $H$ contains two vertices of $V_1$, 
	as follows from the definition of $V_1$.
	Suppose now that $v$ is a vertex of $V_1$. 
	Consider the
	block-cutpoint graph $W$ of $T$.
	Since $v$ has degree 
	at least four in $G$, and no other
	vertex of a block containing $v$ has a neighbor in $V_1$, it follows
	that $v$ either is a cut-vertex of degree at least $3$ in $W$, 
	or that there is a block vertex of degree at least $3$
	in $W$, whose corresponding block in $T$ contains $v$. 
	Moreover, since every block of $T$
	contains at most one vertex from $V_1$,
	we get that for each vertex of $V_1$, there is a unique vertex
	of degree at least $3$ in $W$.
	Now, since $W$ is a tree, it contains more leafs than vertices of degree at least three.
	Hence, the number of leaf blocks of $T$ is greater than $|V_1 \cap V(T)|$.
\end{proof}

	We now define a directed graph $D$ in the following way.
	Each component $T$ in $H$ is represented by a vertex $t$,
	and there is a directed edge from a vertex $t$ to another vertex $t'$
	if there is a vertex $p$ adjacent to $\Delta(G)-1$ vertices of a leaf block
	in $T$ that is also adjacent to a vertex 
	of $T'$. Moreover, for every leaf-block of a component $T$
	of $H$ for which the adjacent vertex $p \in P$ is not adjacent to any other vertex
	in $H$,
	we add a new vertex $x_p$ and an arc $(t,x_p)$;
	thus $x_p$ has indegree $1$ and outdegree $0$ in the digraph $D$.
	Furthermore, every directed edge of $D$ corresponds to a vertex in $P$.

	By Claim \ref{cl:bad} every vertex of $D$ corresponding to a component of
	$H$ has larger outdegree than indegree, so
	it follows from Lemma \ref{lem:digraph} that there is an acyclic directed
	subgraph $D'$ of $D$ such that every vertex corresponding to a component of
	$H$ has outdegree $1$; denote these vertices by $S$.
	
	Since $D'$ is acyclic there is a linear order $\prec$ of the vertices of $S$
	such that if $(t,t')$ is a directed edge of $D'$, then $t \prec t'$.
	
	Suppose $S = \{t_1,\dots,t_n\}$, where $t_1 \prec t_2 \prec \dots \prec t_n$.
	We now traverse the vertices $t_1, t_2,\dots$
	according to $\prec$ in the following
	way. Suppose $(t_1,x) \in A(D')$ (where $x$ might be in $S$).
	
	\begin{itemize}
	
	\item If $t_1$ corresponds to a $\varphi'$-bad component and $(t_1,x) \in A(D')$,
	then we recolor the vertex $p_1$ corresponding to the arc $(t_1, x) \in A(D')$
	to obtain a new coloring $\varphi_1$ of $P$. Note that by the construction of
	$D'$ and the order $\prec$, the only component of $H$ that might be
	$\varphi_1$-bad but not $\varphi'$-bad is a component of $H$ corresponding
	to $x$, if there is such a component. 
	Moreover, after this recoloring, $T_1$ is not $\varphi_1$-bad.
	
	\item If $t_1$ is not bad, then we just set $\varphi_1 =\varphi'$
	and continue with the next vertex
	according to the order $\prec$.
	
	\end{itemize}
	
	We continue this process for all vertices of $S$ according to the linear
	order $\prec$ and define $L$-colorings $\varphi_1,\varphi_2, \dots$.
	Then every vertex in $S$ is considered once, and the process terminates
	after a finite number of steps.
	
	Now, since $t_i \prec t_j$ precisely when there is no directed path from $t_j$
	to $t_i$ in $D'$, it follows that no component $T_i$ corresponding to a vertex 
	$t_i \in S$ can be $\varphi_j$-bad for $j \geq i$.
	Hence, after this process terminates when all vertices of $S$ have been considered,
	we obtain a proper $L$-coloring 
	$\varphi_n$ where no component of $H$ is $\varphi_n$-bad.
	This proves part (i).
	
	\bigskip
	
	Let us now prove part (ii). 
	We shall construct a subgraph $F''$ of $G$ for coloring the 
	vertices of $P$ from their lists
	$L$.
	
	Consider the bipartite multigraph $J$ with parts $P$ and $\mathcal{B}$,
	where $\mathcal{B}$ contains a vertex
	for every leaf-block of $H$, and where a vertex
	$p$ of $P$ is joined by an edge to a vertex
	$b \in \mathcal{B}$ for each edge between 
	$p$ and the block corresponding to $b$ in $H$.
	In $J$, every vertex of $P$ has degree at most
	$\Delta(G)$, and every vertex of $\mathcal{B}$ has degree at
	least $\Delta(G)-1$ and at most $\Delta(G)$. 
	Thus, by K\"onig's edge coloring theorem, there is a proper
	$\Delta(G)$-edge coloring of $J$.
	
	We denote by $J'$ the subgraph of $J$ induced by the edges colored $1, 2,3,4$.
	Then $\Delta(J') \leq 4$ and every vertex of $\mathcal{B}$ has
	degree at least $3$ in $J'$.
	Note further that if some vertex $b \in \mathcal{B}$ 
	has degree $3$ in $J'$, then $b$
	corresponds to a leaf block of $H$ that is contained in a component with at
	least two leaf blocks.

	From $J'$ we form a new graph $J''$ with vertex set $V(J')$
	by for every component $T$ of $H$, where
	every leaf-block of $T$ corresponds to a vertex of degree $3$ in $J'$, 
	arbitrarily selecting two vertices
	$b_1$ and $b_2$, corresponding to different leaf-blocks of $T$, and 
	redistributing the edges incident
	with $b_1$ and $b_2$ so that one edge incident with $b_2$ is ``moved to'' $b_1$; 
	thus $b_1$ gets
	degree four in $J''$ and $b_2$ gets degree two. Note that this does not
	change the degrees of vertices in $P$.
	
	Now, by arbitrarily adding edges between pairs of vertices of odd degree 
	in $J''$, we obtain a graph $K$
	where all vertices have even degree at most $4$.
	Thus every component of $K$ is Eulerian and by taking every second edge in an 
	Eulerian circuit of every component
	of $K$ we obtain a subgraph $F_{K}$ where every vertex has degree $0,1$ or $2$.
	Denote the restriction of this subgraph to $J''$ by $F_{J}$,
	and let $F$ be the subgraph of $J'$ corresponding to $F_J$.
	Then every vertex of $F$ has degree at most two and, moreover, 
	every component of $H$ has a leaf block
	that corresponds to a vertex of degree $2$ in $F$.
	
	Next, let $\mathcal{B}_2$ be the set of vertices in $\mathcal{B}$ 
	that have degree $2$ in $F$ and consider the vertex-induced 
	subgraph $F[\mathcal{B}_2 \cup P]$.
  We define $\mathcal{B}_d$ to
	be the subset of $\mathcal{B}_2$ containing all vertices that
	are incident to two different vertices of $P$ in $F[\mathcal{B}_2 \cup P]$.
	Let $F'$ be the subgraph of $F$ induced by $\mathcal{B}_d \cup P$.

	Now, in $F[\mathcal{B}_2 \cup P]$ every vertex of 
	$\mathcal{B}_2 \setminus \mathcal{B}_d$ is incident with 
	two edges with the same endpoint in $P$, so every vertex of $P$
	with a neighbor in $\mathcal{B}_2 \setminus \mathcal{B}_d$,
	has no other neighbor in $F[\mathcal{B}_2 \cup P]$.
	Set 
	$\mathcal{B}_1 = \mathcal{B}_2 \setminus \mathcal{B}_d$,
	and denote by $\mathcal{B}'_1$ the set of vertices in
	$\mathcal{B}_1$ that correspond to leaf blocks in $H$ with at least
	two different neighbors in $P$ in $G$. Let $P'$ be the
	set of vertices in $P$ that are  adjacent to vertices of 
	$\mathcal{B}'_1$ in $F$.

	From $F'$ we form a new graph $F''$ with
	$V(F'')=P$ and where the edge set is constructed iteratively in the following way:
	\begin{itemize}
	
		\item firstly, two vertices of $P$ are adjacent if
	they have a common neighbor in $F'$; 
	
	\item secondly,
	for all vertices $p'$ of $P'$ we do the following: suppose $p'$ is adjacent to $b$
	in $F$. 
	We add an edge between $p'$ and another arbitrary vertex $p'' \in P$
	that in $J$ is adjacent to $b$.
	\end{itemize}
	
	We note the following properties of $F''$.
	
	\begin{itemize}
	
	\item[(a)] Since every vertex of $F$ has degree at most $2$, the same
	holds for the induced subgraph $F''[P\setminus P']$. Thus, by construction,
	every component of $F''$ is unicyclic, i.e., has at most one cycle.
	
	\item[(b)] Every component $T$ of $H$ that is isomorphic to $K_{\Delta(G)}$ 
	satisfies that two vertices in $P$ with neighbors in $T$ are adjacent in
	$F''$.
	
	\item[(c)] If the vertices of a leaf block $B$ of $H$,
	which corresponds to a vertex in $\mathcal{B}_2$, 
	has at least two
	different vertices in $P$ as neighbors in $G$, then 
	there are two vertices $p,p' \in P$, both with at least one neighbor in $B$,
	that are adjacent in $F''$.
	
	\end{itemize}
	
	We now color the vertices of $P$ as follows.
	By (a), the {\em core} of $F''$,
	obtained by successively removing vertices of degree $1$ from $F''$,
	has maximum degree at most $2$.
	Hence $F''$ is properly colorable from the lists $L$ of vertices in $P$.
	We take such a coloring $\varphi$ of $P$ with a minimum 
	number of $\varphi$-bad components in $H$.
	
	Now, consider this coloring of the vertices of $P$ in the graph $G$.
	It follows from Theorem \ref{th:charac} 
	that if a component $T$ of $H$ is $\varphi$-bad,
	then for every leaf-block $B$ of $T$, every vertex of $P$ that
	has a neighbor in $B$ has the same color under $\varphi$.
	Now, since every component of $H$ has a leaf block that corresponds
	to a vertex of $\mathcal{B}_2$, it follows from (c)
	that there is a leaf-block
	$B_L$ in $T$ where $\Delta(G)-1$ vertices are adjacent to the
	same vertex $p$ in $P$; that is, $p$ has at most one
	neighbor outside $B_L$.
	
	Suppose $\varphi(p)=c_1$ and
	$L(p)=\{c_1,c_2,c_3\}$. Since $p$ has at most one neighbor outside
	$B_L$, it follows that by recoloring $p$ by the color $c_2$ 
	we get a new $L$-coloring $\varphi'$ of $P$ such that either 
	$T$ is $\varphi'$-bad (and then the neighbor
	of $p$ not in $B_L$ is in $T$), or $T$ is not $\varphi'$-bad but some other component 
	is $\varphi'$-bad but not $\varphi$-bad, since $\varphi$ is an $L$-coloring
	with a minimum number of $\varphi$-bad components. Now, since coloring $p$
	by the color $c_2$ does not decrease the number of bad components
	and $p$ has at most one neighbor outside $B_L$, 
	we may color $p$ by the color $c_3$
	to obtain a coloring $\varphi''$ with fewer bad components.
	In the case when $p$ has a neighbor outside $T$, this follows from our
	initial observation before the proof of part (i),
	and in the case when $p$ has all neighbors in $T$, then it follows
	from Corollary \ref{cor:leaf}.
	In both cases, this contradicts the choice of $\varphi$ and thus completes the proof
	of the theorem.
\end{proof}

It is natural to ask if it is possible to prove a version of Theorem
\ref{thm:avoidprecol}
where the set $P$ of vertices with shorter lists is allowed to be 
any independent set.

\begin{problem}
\label{prob:ind}
	Let $G$ be a graph with maximum degree $\Delta=\Delta(G) \geq 4$, 
	not containing a copy of $K_{\Delta+1}$,
	and $P$ an independent set. Is there a $k=k(\Delta)< \Delta$
	such that if $L$ is a list assignment where
	\begin{itemize}
	 \item $|L(v)| =k$, if $v \in P$, and
	
		\item $|L(v)| = \Delta$, if $v \in V(G) \setminus P$,
	\end{itemize}
	then $G$ is $L$-colorable?
\end{problem}

Since $P$ is independent, it follows from Theorem \ref{th:charac} that
any minimal counterexample $G$ to this question
must satisfy that all vertices of $V(G) \setminus P$ has degree $\Delta(G)$.

In particular, we are interested in whether
Problem \ref{prob:ind} has a positive solution in the case when then lists of $P$ have size
$|L(v)| = \Delta(G)-1$; in the next section we investigate this question further
for the special case when all lists are chosen from the
set $\{1,\dots, \Delta(G)\}$.


\section{Avoiding colorings}

In this section we consider the problem of coloring a graph $G$ using colors $1,\dots, \Delta(G)$,
subject to the condition that some colors cannot be used on certain vertices.
Let us begin by noting the following reformulation of Theorem \ref{thm:avoidprecol},
which is a generalization of Brooks' theorem for ordinary graph coloring
for the case of $\Delta(G) \geq 4$.

\begin{corollary}
\label{cor:forbid}
	Let $G$ be a connected graph with maximum degree $\Delta(G) \geq 4$ that is not
	complete. Suppose that a set
	$P$ of vertices is assigned a list of forbidden colors.
	If either
	\begin{itemize}
	\item[(i)] each vertex of $P$ has most $\Delta(G)-2$ forbidden 
	colors and any two vertices of $P$
	are at distance at least $4$, or
	
	\item[(ii)] each vertex of $P$ has most $\Delta(G)-3$ forbidden colors
	and any two vertices of $P$
	are at distance at least $3$,
	\end{itemize}
	then there is a proper $\Delta(G)$-coloring of $G$ which respects the
	forbidden colors.
\end{corollary}

An interesting special case of
Problem \ref{prob:ind} is when the lists of $P$ have size
$|L(v)| = \Delta(G)-1$ and all lists use colors from 
$\{1,\dots, \Delta(G)\}$, which would be a useful generalization of
Brooks' theorem along the lines of Corollary \ref{cor:forbid}. 
We note the following partial result
towards such a general theorem.
Recall that a partial coloring $\varphi$ of $G$ is {\em avoidable} 
if there is a proper coloring $f$ of $G$ such that $f(v) \neq \varphi(v)$
for any vertex that is  colored under $\varphi$.

	\begin{proposition}
	\label{prop:ind}
			Let $G$ be a connected graph with maximum degree $\Delta(G) \geq 4$ that is
			not complete, and $\varphi$ a precoloring of
			a subset $P \subseteq V(G)$.
			If $P$ is independent and every vertex in $V(G) \setminus P$
			is either adjacent to at least three vertices from $P$
			or to no vertex from $P$,
			then  $\varphi$ is avoidable.
	\end{proposition}
	This proposition is a special case of a more general result proved below,
	so we postpone its proof.
	Conversely, for the case when each vertex in $V(G) \setminus P$ has a bounded number of
	neighbors in $P$, then we have the following.
	
	\begin{proposition}
			Let $G$ be a connected graph with maximum degree $\Delta(G) \geq 4$ that is
			not complete,
			and $\varphi$ a precoloring of an independent subset $P \subseteq V(G)$.
			Furthermore, suppose that every vertex in $V(G) \setminus P$
			has at most $d_0$ neighbors in $P$, and that every vertex in
			$P$ has neighbors in at most $d_1$ components of $G-P$.
			If $d_0 d_1 < \Delta(G)-1$ and every component of $G-P$ contains
			a leaf block with at least $d_0+1$ distinct neighbors in $P$, then
			then there is a proper $\Delta(G)$-coloring of $G$ that avoids $\varphi$.
	\end{proposition}
	\begin{proof}
		The proof is based on Theorem 
		\ref{th:charac} and uses similar ideas as the proof of Theorem \ref{thm:avoidprecol}, 
		so we only sketch the argument. The basic idea is to find a proper coloring $f$
		of the vertices of $P$ that avoids $\varphi$, and such that no component of 
		$H=G-P$ together with the list assignment
		for the vertices of $V(G) \setminus P$ obtained by setting
		$$L_f(v) = \{1,\dots, \Delta(G)\} \setminus \{f(u) : \text{$u$ is adjacent to $v$}\}$$
			for $v \in V(G) \setminus P$,
		satisfies the conditions in Theorem \ref{th:charac}. Consequently, there is
		a proper $\Delta(G)$-coloring of $G$ avoiding $\varphi$.
		
		If $T$ is a component of $H$ so that $V(T)$ is not $L_f$-colorable for
		some coloring $f$ of $P$, then every leaf-block of $T$ is a
		complete graph $K$ where all vertices have degree $\Delta(G)$ in $G$,
		and which satisfies that $L_f(u) = L_f(v)$ for any two vertices
		$u$ and $v$ of $K$ that are not cut-vertices. Moreover, the cut-vertex $w$
		of $K$ satisfies that $L_f(u) \subseteq L_f(w)$, where $u \in V(K)$.
		Hence, it suffices to prove that there is a coloring $f$ of $P$
		so that every component of $G-P$
		has a leaf block that does not satisfy this condition.
		
		To this end, we define a
	  ''conflict graph'' $D$ with vertex set $P$ 
		in the following way.
		Suppose that
		$K$ is a leaf block in $H$ that is isomorphic to $K_k$, for some integer $k$,
		and where the vertices in $K$ altogether have at least
		$d_0+1$ distinct neighbors in $P$. 
		We arbitrarily pick $d_0+1$ vertices in $P$, each of which has a neighbor in $K$,
		and form a clique on those $d_0+1$ vertices.
		
		Now, by assumption every component of $G-P$ contains such a leaf block as described
		in the preceding paragraph, so
		by repeating this process for one such leaf block of every component
		of $H$, we obtain the conflict graph $D$. The maximum degree of this graph
		is at most $d_1d_0$, since every vertex of $P$ has neighbors in at most
		$d_1$ components of $H$. Thus, by the list coloring version of
		Brooks' theorem, there is a proper $\Delta(G)$-coloring $f$ of $D$ that
		avoids $\varphi$.
		
		Now, by the construction of $D$ every component of $H$ has a 
		leaf block with 
		neighbors in $P$ of $d_0+1$
		different colors under $f$.
		Since every vertex of $V(G) \setminus P$ has at most $d_0$ neighbors
		in $P$, this implies that no component of $H$ together with the list assignment 
		$$L_f(v) = \{1,\dots, \Delta(G)\} \setminus \{f(u) : \text{$u$ is adjacent to $v$}\}$$
		satisfies the conditions in Theorem \ref{th:charac}
		and, consequently, there is a proper $\Delta(G)$-coloring of $G$ that avoids $\varphi$.
	\end{proof}

More generally, we might ask what can be said in the case when $P$ is not required
to be an independent set? A $\Delta(G)$-coloring of
$G$ where every vertex is assigned the same color is not avoidable if $\chi(G)= \Delta(G)$.
We note the following reformulation of a result for edge colorings from
\cite{CasselgrenJohanssonMarkstrom}. 

\begin{proposition}
\label{prop:gen}
	If $\varphi$ is a partial $k$-coloring of a connected graph $G$ that is not complete,
	where every color
	appears on at most $\Delta(G)-k$ vertices, then there is a proper $\Delta(G)$-coloring
	of $G$ that avoids $\varphi$.
\end{proposition}
The proof is virtually identical to the corresponding result for edge colorings
in \cite{CasselgrenJohanssonMarkstrom}, so we omit it.
Note, in particular, that Proposition \ref{prop:gen} implies that
every partial coloring with at most $\Delta(G)-1$ colored vertices is avoidable,
which is sharp by the example of the join of $K_{d}$ with two vertices
$u$ and $v$, where $u$ and all vertices of $K_d$ are colored $1$.

For the case when the vertices in $P$ induce a subgraph of small maximum degree,
we have the following.  
	
\begin{proposition}
\label{prop:prop}
			Let $G$ be a connected graph with maximum degree
			$\Delta(G) \geq 3$ that is not complete,
			and $\varphi$ a precoloring of a subset $P \subseteq V(G)$,
			where every vertex in $V(G) \setminus P$ either has at least $d$ neighbors
			in $P$ or no neighbors in $P$.
			If $\Delta(G[P]) < d-2$, then there is a proper coloring of $G$
			avoiding $\varphi$.
\end{proposition}
\begin{proof}
	Again, the argument is similar to the one in the proof
	of Theorem \ref{thm:avoidprecol},
	so we only sketch the proof.
	
	Since $\Delta(G[P]) < d-2$, there is a proper $(d-1)$-coloring $f$ of
	$G[P]$ that avoids $\varphi$. 
	Now, consider the lists induced 
		for the vertices of $V(G) \setminus P$ obtained by setting
	$$L_f(v) = \{1,\dots, \Delta(G)\} \setminus \{f(u) : \text{$u$ is adjacent to $v$}\}$$
	for $v \in V(G) \setminus P$.
	
	Since every vertex of $V(G) \setminus P$
	has at least $d$ neighbors in $P$ or no neighbor in $P$, no component $T$ of the graph $G-P$ 
	along with the restriction of  $L_f$ to $T$
	satisfies the conditions in Theorem \ref{th:charac}. Hence, there is an $L_f$-coloring
	of $G-P$. This $L$-coloring taken together with $f$ is the required coloring.
	\end{proof}
Note that Proposition \ref{prop:ind} is a special 
case of Proposition \ref{prop:prop}.

\bigskip

Next, we consider the special case that the partial coloring to be avoided is proper.
It is straightforward that for the case $\Delta(G) \leq 2$, there
are unavoidable partial proper colorings of even cycles and paths, while any partial
proper $3$-coloring of an odd cycle is avoidable.
For the case $\Delta(G)=3$, the examples after Theorem \ref{thm:avoidprecol}
show that there are unavoidable partial proper $3$-colorings of graphs with maximum degree three.
For the case $\Delta(G) =4$, Figure \ref{fig:unavoidable}
depicts a
graph with an unavoidable partial proper coloring. For $\Delta(G) \geq 5$
we do not know any such examples, and in general we would like to suggest
the following.

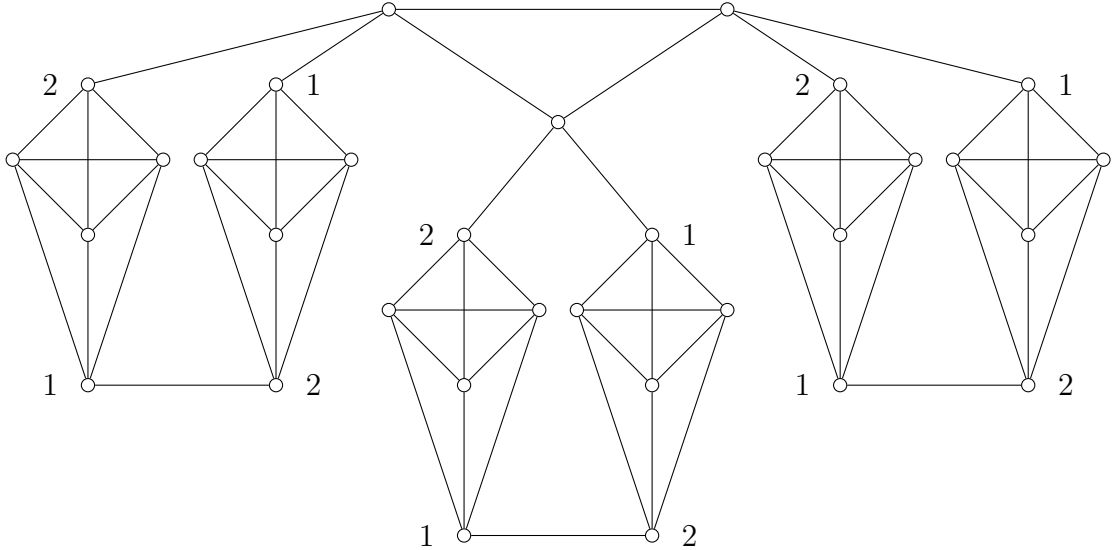
\begin{figure} [h]
	\begin{center}

	\begin{tikzpicture}[scale=0.5]
	\tikzset{vertex/.style = {shape=circle,draw,
												inner sep=0pt, minimum width=5pt}}
	\tikzset{edge/.style = {-,> = latex'}}

	\node[vertex] (u_1) at  (0, -4) {};
	\node at (-1,-4) {$1$};
		\node[vertex] (u_2) at  (5, -4) {};
\node at (6,-4) {$2$};
	
	\node[vertex] (x_1) at  (0, 0) {};
	\node[vertex] (x_2) at  (-2, 2) {};
	\node[vertex] (x_3) at  (2, 2) {};
	
	\node[vertex] (y_1) at  (5, 0) {};
	\node[vertex] (y_2) at  (3, 2) {};
	\node[vertex] (y_3) at  (7, 2) {};

	\node[vertex] (z_1) at  (0, 4) {};
	\node at (-1,4) {$2$};
	\node[vertex] (z_2) at  (5, 4) {};
	\node at (6,4) {$1$};
	
		\node[vertex] (w_1) at  (2.5, 7) {};

			\draw[edge] (u_1) to node [] {}  (u_2);

	\draw[edge] (u_1) to node [] {}  (x_1);	
	\draw[edge] (u_1) to node [] {}  (x_2);	
	\draw[edge] (u_1) to node [] {}  (x_3);

	\draw[edge] (x_1) to node [] {}  (x_2);	
	\draw[edge] (x_3) to node [] {}  (x_2);	
	\draw[edge] (x_1) to node [] {}  (x_3);

	\draw[edge] (y_1) to node [] {}  (y_2);	
	\draw[edge] (y_3) to node [] {}  (y_2);	
	\draw[edge] (y_1) to node [] {}  (y_3);

	\draw[edge] (u_2) to node [] {}  (y_1);	
	\draw[edge] (u_2) to node [] {}  (y_2);	
	\draw[edge] (u_2) to node [] {}  (y_3);

	\draw[edge] (z_1) to node [] {}  (x_1);	
	\draw[edge] (z_1) to node [] {}  (x_2);	
	\draw[edge] (z_1) to node [] {}  (x_3);
	
	\draw[edge] (z_2) to node [] {}  (y_1);	
	\draw[edge] (z_2) to node [] {}  (y_2);	
	\draw[edge] (z_2) to node [] {}  (y_3);
	
	\draw[edge] (z_1) to node [] {}  (w_1);
	\draw[edge] (z_2) to node [] {}  (w_1);
%
%
	\node[vertex] (u_11) at  (-10, 0) {};
	\node at (-11,0) {$1$};
		\node[vertex] (u_12) at  (-5, 0) {};
\node at (-4,0) {$2$};
	
	\node[vertex] (x_11) at  (-10, 4) {};
	\node[vertex] (x_12) at  (-12, 6) {};
	\node[vertex] (x_13) at  (-8, 6) {};
	
	\node[vertex] (y_11) at  (-5, 4) {};
	\node[vertex] (y_12) at  (-7, 6) {};
	\node[vertex] (y_13) at  (-3, 6) {};

	\node[vertex] (z_11) at  (-10, 8) {};
	\node at (-11,8) {$2$};
	\node[vertex] (z_12) at  (-5, 8) {};
	\node at (-4,8) {$1$};
	
		\node[vertex] (w_2) at  (-2, 10) {};

			\draw[edge] (u_11) to node [] {}  (u_12);

	\draw[edge] (u_11) to node [] {}  (x_11);	
	\draw[edge] (u_11) to node [] {}  (x_12);	
	\draw[edge] (u_11) to node [] {}  (x_13);

	\draw[edge] (u_12) to node [] {}  (y_11);	
	\draw[edge] (u_12) to node [] {}  (y_12);	
	\draw[edge] (u_12) to node [] {}  (y_13);

	\draw[edge] (x_11) to node [] {}  (x_12);	
	\draw[edge] (x_13) to node [] {}  (x_12);	
	\draw[edge] (x_11) to node [] {}  (x_13);

	\draw[edge] (y_11) to node [] {}  (y_12);	
	\draw[edge] (y_13) to node [] {}  (y_12);	
	\draw[edge] (y_11) to node [] {}  (y_13);
	
	\draw[edge] (z_11) to node [] {}  (x_11);	
	\draw[edge] (z_11) to node [] {}  (x_12);	
	\draw[edge] (z_11) to node [] {}  (x_13);
	
	\draw[edge] (z_12) to node [] {}  (y_11);	
	\draw[edge] (z_12) to node [] {}  (y_12);	
	\draw[edge] (z_12) to node [] {}  (y_13);
	
	\draw[edge] (z_11) to node [] {}  (w_2);
	\draw[edge] (z_12) to node [] {}  (w_2);
%
%
	\node[vertex] (u_21) at  (10, 0) {};
	\node at (9,0) {$1$};
		\node[vertex] (u_22) at  (15, 0) {};
\node at (16,0) {$2$};
	
	\node[vertex] (x_21) at  (10, 4) {};
	\node[vertex] (x_22) at  (8, 6) {};
	\node[vertex] (x_23) at  (12, 6) {};
	
	\node[vertex] (y_21) at  (15, 4) {};
	\node[vertex] (y_22) at  (13, 6) {};
	\node[vertex] (y_23) at  (17, 6) {};

	\node[vertex] (z_21) at  (10, 8) {};
	\node at (9,8) {$2$};
	\node[vertex] (z_22) at  (15, 8) {};
	\node at (16,8) {$1$};
	
		\node[vertex] (w_3) at  (7, 10) {};

			\draw[edge] (u_21) to node [] {}  (u_22);

	\draw[edge] (u_21) to node [] {}  (x_21);	
	\draw[edge] (u_21) to node [] {}  (x_22);	
	\draw[edge] (u_21) to node [] {}  (x_23);

	\draw[edge] (u_22) to node [] {}  (y_21);	
	\draw[edge] (u_22) to node [] {}  (y_22);	
	\draw[edge] (u_22) to node [] {}  (y_23);

	\draw[edge] (x_21) to node [] {}  (x_22);	
	\draw[edge] (x_23) to node [] {}  (x_22);	
	\draw[edge] (x_21) to node [] {}  (x_23);

	\draw[edge] (y_21) to node [] {}  (y_22);	
	\draw[edge] (y_23) to node [] {}  (y_22);	
	\draw[edge] (y_21) to node [] {}  (y_23);
	
	\draw[edge] (z_21) to node [] {}  (x_21);	
	\draw[edge] (z_21) to node [] {}  (x_22);	
	\draw[edge] (z_21) to node [] {}  (x_23);
	
	\draw[edge] (z_22) to node [] {}  (y_21);	
	\draw[edge] (z_22) to node [] {}  (y_22);	
	\draw[edge] (z_22) to node [] {}  (y_23);
	
	\draw[edge] (z_21) to node [] {}  (w_3);
	\draw[edge] (z_22) to node [] {}  (w_3);
	
	\draw[edge] (w_1) to node [] {}  (w_3);
	\draw[edge] (w_1) to node [] {}  (w_2);
	\draw[edge] (w_2) to node [] {}  (w_3);

\end{tikzpicture}
\end{center}
\caption{A graph with a partial proper coloring that is not avoidable.}
	\label{fig:unavoidable}
\end{figure}

\begin{problem}
	\label{prob:avoid}
	Is there a constant $\Delta_0$ such that if $\Delta(G) \geq \Delta_0$
	and $G$ does not contain a copy of $K_{\Delta(G)+1}$,
	then any partial proper $\Delta(G)$-coloring of $G$ is avoidable?
\end{problem}

	We note that every proper coloring of a graph is avoidable by permuting colors;
	thus any extendable partial coloring is avoidable.

	\bigskip
		
Finally, let us briefly consider corresponding questions for $k$-colorings of
$k$-chromatic graphs; that is, if $G$ is a $k$-colorable graph
and every vertex is assigned a forbidden color, can we find a proper coloring
that avoids the forbidden colors?

Since for any $k$, there are well-known examples of uniquely $k$-colorable
graphs, we need at least $k+1$ colors for such a coloring.
Moreover, since there are examples of planar graphs
with $4$-list assignments with colors from the set $\{1,\dots, 5\}$,
that are not list colorable (see e.g. \cite{Albertson}), 
it is in general not possible to avoid a
$(k+1)$-coloring of a $k$-colorable graph using $k+1$ colors. 
Furthermore, in the example from
\cite{Albertson},
the missing colors in fact form a proper coloring of the planar graph,
so it follows that not every proper $5$-coloring of a planar graph is avoidable.

On the other hand, if we require that the set of vertices
with a forbidden color form an independent set, then any partial $(k+1)$-coloring
of a $k$-chromatic graph is trivially avoidable.

\begin{proposition}
	If $G$ is a $k$-colorable graph, and $\varphi$ a partial $(k+1)$-coloring
	of an independent set in $G$, then there is a proper $(k+1)$-coloring
	of $G$ avoiding $\varphi$.
\end{proposition}

\begin{proof}
	Let $f$ be any proper $k$-coloring of $G$. For any vertex $v$
	that has the same color under $f$ and $\varphi$ we recolor
	$v$ by the color $k+1$; the resulting coloring is proper and avoids $\varphi$.
\end{proof}


\begin{thebibliography}{99}
\addcontentsline{toc}{chapter}{Bibliography}

\bibitem{Albertson}
M.O. Albertson,
You Can't Paint yourself into a corner,
J. Combin. Theory Ser. B 78 (1998), 189--194.

\bibitem{AlbertsonKostochkaWest}
M.O. Albertson, A.V. Kostochka, D.B. West,
Precoloring extensions of Brooks' theorem,
SIAM J. Discriete Math. 18 (2005), 542--553.

\bibitem{Axenovich}
M. Axenovich
A note on graph coloring extensions and list-colorings, Electron. J. Combin. 10 (2003),
note 1.


\bibitem{Borodin}
O.V. Borodin, Criterion of chromaticity of a degree prescription,
Abstracts of IV All-Union Conference on Theoretical Cybernetics
(Novisibirsk), 1977, 127--128 (in Russian).


\bibitem{CasselgrenPham} 
C.J. Casselgren, L.A. Pham,
Restricted extension of sparse partial edge colorings of complete
graphs. Electron. J. Combin. 28 (2021).


\bibitem{Casselgren} 
C.J. Casselgren,
On avoiding some families of arrays. Discr. Math. 312 (2012), 963--972.


\bibitem{CasselgrenJohanssonMarkstrom}
C.J. Casselgren, P. Johansson, K. Markstr\"om,
Avoiding and Extending Partial Edge Colorings
of Hypercubes, Graphs and Combin. 38 (2022).




\bibitem{CranstonRabern}
D.W. Cranston, L. Rabern,
Brooks' Theorem and Beyond,
J. Graph Theory 80 (2015), 199--225.

\bibitem{EGHKPS}
K. Edwards, A. Girao, J. van den Heuvel, R.J. Kang, G.J. Puleo,  
J.-S. Sereni,
Extension from Precoloured Sets of Edges,
{\em Electron J. of Combin.} 25 (2018), P3.1, 28 pp.

\bibitem{ERT}
P. Erd\H os, A.L. Rubin, H. Taylor, Choosability in graphs,
Proceedings of the West Coast Conference on Combinatorics, Graph Theory
and Computing, Utilitas Math, Winnipeg, MB, Canada, 1980, 126--157.


\bibitem{Haggkvist}
R. H\"aggkvist,
A note on Latin squares with restricted support. 
Discr. Math. 75 (1989), 253--254.

\bibitem{Vizing}
V.G. Vizing, 
Vextex coloring with given colors, 
Metody Diskretn. Anal. 29 (1976), 3--10 (in Russian).

\bibitem{Voigt1}
M. Voigt, Precoloring extension for $2$-connected Graphs with maximum degree three,
Discrete Math. 309 (2009), 4926--4930.


\bibitem{Voigt2}
M. Voigt,
Precoloring extension for 2-connected graphs. SIAM J. Discrete Math. 21 (2007), 258--263.


\end{thebibliography}
\end{document}